\documentclass[11pt]{article}
\usepackage[T1]{fontenc}
\usepackage[utf8]{inputenc}
\usepackage{geometry}
\usepackage{amsmath}
\usepackage{amsthm}
\usepackage{enumitem}
\usepackage{setspace}
\usepackage[unicode=true,
 bookmarks=false,
 breaklinks=false,pdfborder={0 0 1},backref=false,colorlinks=false]
 {hyperref}
\usepackage[dvipsnames]{xcolor} 
\usepackage[color=Purple!50!white, textwidth=22mm]{todonotes}
\usepackage{comment}
\allowdisplaybreaks 

\makeatletter
\theoremstyle{plain}
\newtheorem{thm}{\protect\theoremname}
\theoremstyle{plain}
\newtheorem{lem}[thm]{\protect\lemmaname}
\theoremstyle{plain}
\newtheorem{conj}[thm]{\protect\conjname}
\theoremstyle{plain}
\newtheorem{claim}[thm]{\protect\claimname}
\theoremstyle{remark}
\newtheorem{remark}[thm]{Remark}
\usepackage{amssymb}

\makeatother

\providecommand{\claimname}{Claim}
\providecommand{\lemmaname}{Lemma}
\providecommand{\theoremname}{Theorem}
\providecommand{\conjname}{Conjecture}

\setenumerate[0]{label=(\roman*), ref=\roman*}

\begin{document}
\title{The upper logarithmic density of monochromatic subset sums}
\author{David Conlon\thanks{Department of Mathematics, California Institute of Technology, Pasadena, CA 91125. Email:  {\tt dconlon@caltech.edu.} Research supported by NSF Award DMS-2054452.} \and Jacob Fox\thanks{Department of Mathematics, Stanford University, Stanford, CA 94305. Email: {\tt jacobfox@stanford.edu}. Research supported by a Packard Fellowship and by NSF Awards DMS-1855635 and DMS-2154169.} \and Huy Tuan Pham\thanks{Department of Mathematics, Stanford University, Stanford, CA 94305. Email: {\tt huypham@stanford.edu}. Research supported by a Two Sigma Fellowship.}}
\date{}
\maketitle

\begin{abstract}
We show that in any two-coloring of the positive integers there is a color for which the set of positive integers that can be represented as a sum of distinct elements with this color has upper logarithmic density at least $(2+\sqrt{3})/4$ and this is best possible. This answers a forty-year-old question of Erd\H{o}s. 
\end{abstract}

\section{Introduction}

For a set $A$ of positive integers, the logarithmic density $d_\ell(A; x)$ of $A$ up to $x$ is $\frac{1}{\log x}\sum_{a \in A, a \leq x}1/a$, where $\log x$ denotes the natural logarithm of $x$. The {\it upper logarithmic density} of $A$ is then $\bar{d}_{\ell}(A)=\limsup_{x \to \infty} d_\ell(A; x)$. Such logarithmic density functions arise very naturally in number theory. For instance,  a classical result of Davenport and Erd\H{o}s~\cite{DE36} (see also~\cite{HR83}) shows that any set of positive integers $A$ with positive upper logarithmic density contains an infinite division chain, that is, an infinite sequence $a_{i_1} < a_{i_2} < \cdots$ with $a_{i_j} \in A$ and $a_{i_j} \mid a_{i_{j+1}}$ for all $j \geq 1$. Much more recently, the celebrated Erd\H{o}s discrepancy problem was settled by Tao~\cite{T16} using his progress~\cite{T162} on a logarithmically-averaged version of the Elliott conjecture on the distribution of bounded multiplicative functions.

Our concern here will be with a problem of Erd\H{o}s concerning subset sums. Given a set of integers $A$, the \emph{set of subset sums} $\Sigma(A)$ is the set of all integers that can be represented as a sum of distinct elements from $A$. That is, 
\[\Sigma(A) = \left\{ \sum_{s \in S} s : S \subseteq A \right\}.\]
Suppose now that $r \geq 2$ is an integer and consider a partition $\mathbb{N}=A_1 \sqcup \cdots \sqcup A_r$ of the positive integers into $r$ parts. In the problem papers~\cite{Er82, Er82a}, Erd\H{o}s noted that there must then be some $i \in [r]$ such that the upper density of $\Sigma(A_i)$ is $1$ and the upper logarithmic density of $\Sigma(A_i)$ is at least $1/2$. He also observed that if $A_2$ consists of those $n$ for which $\lfloor \log_4 \log_2 n \rfloor$ is even and $A_1$ is the complement of $A_2$, then the upper logarithmic density of both $\Sigma(A_1)$ and $\Sigma(A_2)$ is less than one. In fact, one can check that in this example each of $\Sigma(A_1)$ and $\Sigma(A_2)$ has upper logarithmic density 
$14/15$.\footnote{Erd\H{o}s incorrectly implies in \cite{Er82} that in his construction the upper logarithmic density of both $\Sigma(A_i)$ is at most $3/4$.} Following this line of inquiry to its natural end, Erd\H{o}s~\cite{Er82, Er82a} asked for a determination of $c_2$, the largest real number such that every two-coloring of the positive integers has a color class such that the upper logarithmic density of its set of subset sums is at least $c_2$.

More generally, let $c_r$ be the minimum, taken over all partitions of $\mathbb{N}$ into $r$ parts $A_1,\ldots,A_r$, of the maximum over $i=1,\dots,r$ of the upper logarithmic density of $\Sigma(A_i)$. That is, 
$$c_r=\min_{\mathbb{N}=A_1 \sqcup \cdots \sqcup A_r}\max_{i \in [r]}\bar{d}_{\ell}(\Sigma(A_i)).$$
Here we give a general upper bound for $c_r$ and, answering Erd\H{o}s' question, show that it is tight for $r = 2$. 
We suspect that our upper bound is also tight for all $r \geq 3$, but our methods do not seem sufficient for proving this. We refer the reader to the brief concluding remarks for a little more on this issue.

\begin{thm} \label{thm:main}
For any integer $r \geq 2$, $c_r$ is at most 
$$\left(1-\frac{1}{2b_0}\right)\left(1+\frac{1}{2rb_0-r}\right),$$ 
where $b_0$ is the unique root of the polynomial $b^r-2rb+r-1$ with $b > 1$, and this is tight for $r = 2$, where $c_2=(2+\sqrt{3})/4 \approx 0.93301$.
\end{thm}

We start with the upper bound, which is comparatively simple, following as it does from an appropriate generalization of Erd\H{o}s' coloring.
Indeed, fix an integer $r \geq 2$ and a real number $b > 1$ 
and consider the $r$-coloring of the positive integers where $n$ is given the value of $\lfloor \log_b \log n \rfloor$ taken modulo $r$. 
Erd\H{o}s' coloring mentioned earlier is essentially the special case where $r=2$ and $b=4$. 
Using the observation that the set of non-zero subset sums of the interval $[m, n]$ is contained in the interval $[m, \binom{n+1}{2}]$, 
it is easily checked that the upper logarithmic density of the set of subset sums of each color class is at most
$$\delta_r(b)=\left(1-\frac{1}{2b}\right)(1+b^{-r}+b^{-2r}+\cdots)=\left(1-\frac{1}{2b}\right)\left(1-b^{-r}\right)^{-1}.$$
Since $c_r \leq \delta_r(b)$ for any $b > 1$, we wish now to minimize $\delta_r(b)$. To this end, note that the derivative 
of $\delta_r(b)$ with respect to $b$ is
$$\delta'_r(b)=\frac{1}{2b^2}\left(1-b^{-r}\right)^{-1}-\left(1-\frac{1}{2b}\right)rb^{-r-1}\left(1-b^{-r}\right)^{-2}.$$
The minimum value of $\delta_r(b)$ occurs when this equals zero or, simplifying, when $b^r-2rb+r-1=0$. By Descartes' rule of signs, this polynomial has at most two positive roots and it is easily checked that there are precisely two roots, one lying between $0$ and $1$ and the other lying above $1$, thus completing the proof of the claimed upper bound.

We now turn our attention to our main contribution, the proof of the lower bound for $c_2$, which ultimately relies on an application of the Brouwer fixed-point theorem. We begin by proving a crucial lemma about monochromatic subset sums which may be of independent interest.

\section{Intervals of monochromatic subset sums}

In this section, we use a result from our recent paper~\cite{CFP} to prove the following key lemma on subset sums, which will be important in the proof of the lower bound for $c_2$. We note that a weaker version of this lemma, from which the bound $c_r \geq 1/2$ easily follows, was previously claimed by Erd\H{o}s~\cite[Theorem 3]{Er82a}, though the proof of this statement was never published.

\begin{lem}\label{mainlem}
For every positive integer $r$, there are positive constants $C=C(r)$ and $C'=C'(r)$ such that the following holds. For every $N>0$ and every partition $\mathbb{N} \cap [N,eN)=A_1 \sqcup A_2 \sqcup \cdots \sqcup A_r$ into $r$ color classes,  
there is some $i \in [r]$ such that $\Sigma(A_i)$ contains all positive integers in $[CN,C'N^2]$. 
\end{lem}

To state the result from~\cite{CFP} that we need for the proof of this lemma, we introduce the notation
\[
\Sigma^{[k]}(A) = \left\{\sum_{s\in S}s\,\,:\,\,S\subseteq A, \, |S|\le k\right\}.
\]
That is, $\Sigma^{[k]}(A)$ is the set of subset sums formed by adding at most $k$ distinct elements of $A$.

\begin{thm}[Theorem 6.1 of \cite{CFP}] \label{thm:CFP} 
There exists an absolute constant $C>0$ such that the following holds. For any subset $A$ of $[n]$ of size $m\ge C\sqrt{n}$,
there exists $d\ge1$ such that, for $A'=\{x/d\,:\,x\in A, \, d|x\}$ and $k = 2^{50}n/m$, $\Sigma^{[k]}(A')$ contains an interval of length at least $n$.
Furthermore, 
\[
|A|-|A'|\le 2^{30}(\log n)^{3}+\frac{2^{30}n}{m}.
\]
\end{thm}

We will also use the following observation of Graham~\cite{G}. Part (ii), which is the part we will use, follows from the elementary part (i) by induction.

\begin{lem}[Graham~\cite{G}]\label{lem:verysimple}
Let $A$ be a set such that $\Sigma(A)$ contains all integers in the interval $[x,x+y)$.
\begin{enumerate}
\item If $a$ is a positive integer with $a \leq y$ and $a\notin A$, then $\Sigma(A \cup \{a\})$ contains all integers in the interval $[x,x+y+a)$. 
\item If $a_1,\ldots,a_s$ are positive integers such that $a_i \leq y+\sum_{j<i} a_j$ and $a_i \notin A$ for $i=1,\ldots,s$, then  
$\Sigma(A \cup \{a_1,a_2,\dots,a_s\})$ contains all integers in the interval $[x,x+y+\sum_{i=1}^s a_i)$.
\end{enumerate}
\end{lem}

We are now in a position to prove Lemma~\ref{mainlem}. Recall that a \emph{homogeneous progression} is an arithmetic progression $a, a+ d, \dots, a+kd$, where $d$ divides $a$ and, hence, every other term in the progression.

\begin{proof}[Proof of Lemma \ref{mainlem}]
Suppose, without loss of generality, that $r$ is sufficiently large and $N$ is sufficiently large in terms of $r$. Let $X$ be the set of elements of $\mathbb{N}\cap [N,eN)$ which do not have any prime factor at most $r^2$. Let $W=\prod_{p\le r^2} p$ and note that the number of integers in an interval of length $\ell$ which are coprime to $W$ is at least $(1-o_\ell(1))\ell \phi(W)/W$, where $\phi$ is the Euler totient function. By Merten's third theorem, $\phi(W)/W = (e^{-\gamma}+o(1))/\log(r^2) \ge 1/(3.9 \log r)$ for $r$ sufficiently large, where $\gamma$ is the Euler--Mascheroni constant. Thus, as $N$ is sufficiently large in terms of $r$, we have 
\[|X| \ge (e-1)N\cdot 1/(4\log r) \ge N/(4\log r).\]
Therefore, by the pigeonhole principle, there exists an index $i$ such that $|A_i\cap X| \ge N/(4r\log r)$. Fix such an $i$ and let $A$ be an arbitrary subset of $A_i\cap X$ of size $N/(8r\log r)$. 

By Theorem \ref{thm:CFP}, there exists $d\ge 1$ and a subset $A^*$ of $A$ consisting of multiples of $d$ such that 
$$|A^*| \ge |A|-2^{30}(\log (eN))^{3} - \frac{2^{30}eN}{|A|} \ge |A|/2 \ge N/(16r\log r)$$ 
and, for $k=2^{50}eN/|A|$, $\Sigma^{[k]}(A^*)$ contains a homogeneous progression of common difference $d$ and length at least $eN$. 
If $d>1$, then, since $A$ does not contain multiples of any prime $p\le r^2$, we must have $d\ge r^2$. But then $|A^*| \le 1+eN/r^2 < N/(16r\log r)$, a contradiction. We must therefore have that $d=1$ and, hence, $\Sigma^{[k]}(A^*)$ contains an interval $I$ of length at least $eN$. 

Since $k=2^{50}eN/|A| \le 2^{55}r\log r$, the smallest element of $I$ is at most $2^{55}r\log r \cdot eN < 2^{57}Nr \log r$. Therefore, by Lemma \ref{lem:verysimple}, we see that $\Sigma(A^*\cup (A_i\setminus A))$ contains all integers between $2^{57}Nr \log r$ and $\sum_{x\in A_i \setminus A} x \ge N^2/(8r\log r)$, as required. 
\end{proof}

\begin{remark}
Alternatively, one can prove Lemma~\ref{mainlem} by using Theorem 7.1 from Szemer\'edi and Vu's paper~\cite{SV2}. This result says that there is a constant $C > 0$ such that if $A$ is a subset of $[n]$ of size $m \geq C\sqrt{n}$ and $k\ge Cn/m$, then $\Sigma^{[k]}(A)$ contains an arithmetic progression of length at least $n$. If we apply this result rather than Theorem~\ref{thm:CFP} to the set $A$, we find an arithmetic progression rather than an interval in $\Sigma(A)$. However, we may then use the fact that the elements of $A_i \cap X$ do not have small factors to expand this arithmetic progression to an interval. The remainder of the proof then proceeds as before.
\end{remark}

\section{Proof of the lower bound for $c_2$} \label{sec:lower}

Suppose $\mathbb{N} = A_1 \sqcup \cdots \sqcup A_r$ is a partition of the positive integers into $r$ color classes.  Given this partition, we build an auxiliary $r$-coloring $\alpha:\mathbb{N} \rightarrow [r]$ of the positive integers, where we set $\alpha(n)=i$ for some $i$ such that the color class $A_i$ of integers colored by $i$ has the property that $\Sigma(A_i)$ contains all positive integers in the interval $[CN,C'N^2]$, where $N = e^n$ and $C$ and $C'$ are as in Lemma~\ref{mainlem}. Note that at least one choice for $i$ always exists by Lemma~\ref{mainlem}. 

From this auxiliary coloring $\alpha$, we build another auxiliary coloring $\phi : \mathbb{N} \rightarrow 2^{[r]}$ of the positive integers, where each positive integer now receives a set of at least one color. Explicitly, we place $i$ in $\phi(n)$ if and only if there is some $n/2 \leq j \leq n$ such that $\alpha(j)=i$. Let $S(\phi,i)$ be the set of positive integers $n$ such that $i \in \phi(n)$. The next lemma shows that the upper logarithmic density of $\Sigma(A_i)$ is at least the upper density of $S(\phi,i)$. 

\begin{lem} \label{lem:denred}
The upper logarithmic density of $\Sigma(A_i)$ is at least the upper density of $S(\phi,i)$. 
\end{lem}

\begin{proof}
Let $\gamma$ be a sufficiently large constant depending on $C$ and $C'$, where again $C$ and $C'$ are as in Lemma~\ref{mainlem}. Consider the coloring $\tilde{\phi}:\mathbb{N} \rightarrow 2^{[r]}$ such that $i$ is in $\tilde{\phi}(n)$ if and only if there is some $n/2+\gamma\le j\le n-\gamma$ such that $\alpha(j)=i$. Then, if $n \in S(\tilde{\phi},i)$, there exists $j \in [n/2+\gamma,n-\gamma]$ such that $\alpha(j)=i$ and so, by definition, $\Sigma(A_i)$ contains $[Ce^{j},C'e^{2j}]$. Hence, for $\gamma$ sufficiently large, $\Sigma(A_i)$ contains $[e^{n},e^{n+1}]$. Noting that $\sum_{e^n\le x<e^{n+1}} 1/x = 1 + O(e^{-n})$, we obtain that the upper logarithmic density of $\Sigma(A_i)$ is at least the upper density of $S(\tilde{\phi},i)$. 

It remains to prove that the upper density of $S(\tilde{\phi},i)$ is at least the upper density of $S(\phi,i)$. Partition the elements of $S(\tilde{\phi},i)$ into disjoint intervals $I_k$, so that $\min(I_{k+1}) > 1+\max(I_k)$ for any $k\ge 1$. Observe that $S(\tilde{\phi},i)$ is the union of intervals of the form $[j+\gamma,2j-2\gamma]$ where $\alpha(j)=i$. Thus, we must have $|I_k| \ge k-3\gamma$. Similarly, $S(\phi,i)$ is the union of intervals of the form $[j,2j]$ where $\alpha(j)=i$. Let $S_1(\phi,i)$ be the union of those intervals $[j,2j]$ with $\alpha(j)=i$ and $j \le 3\gamma$ and let $S_2(\phi,i) = S(\phi,i)\setminus S_1(\phi,i)$. Observe that if $x\in S(\phi,i) \setminus S(\tilde{\phi},i)$, then either $x\in S_1(\phi,i)$ or there exists $k$ such that $x\in [\min I_k-\gamma, \min I_k) \cup (\max I_k, \max I_k + 2\gamma]$. Let $t$ be a sufficiently large positive integer and let $\ell$ be the number of intervals $I_k$ intersecting $[t]$. We then have that $|(S(\phi,i) \cap [t]) \setminus (S(\tilde{\phi},i) \cap [t])| \le 3 \gamma (\ell+1) + (3\gamma+1)^2/2$, where we used that $|S_1(\phi,i)| \le \sum_{j\le 3\gamma}j < (3\gamma+1)^2/2$. Using that $|I_k| \ge k - 3\gamma$, we have $\ell + 1 \le 2\sqrt{t}$ for $t$ sufficiently large and, hence,  
\[
\frac{|S(\phi,i) \cap [t]|}{t}-\frac{|S(\tilde{\phi},i) \cap [t]|}{t} \le \frac{6\gamma}{\sqrt{t}} + \frac{(3\gamma+1)^2}{2t} \le \frac{12\gamma}{\sqrt{t}}. 
\]
Thus, 
\[
\limsup_{t\to \infty}\frac{|S(\phi,i) \cap [t]|}{t}=\limsup_{t\to \infty}\frac{|S(\tilde{\phi},i) \cap [t]|}{t}. \qedhere
\]
\end{proof}

The next lemma therefore completes the proof of the lower bound for $c_2$ by showing that, for $r = 2$, the upper density of $S(\phi,i)$ is at least $f_2 := \inf_{z\in [0,1)} \frac{1-z/2}{1-z^2} = (2+\sqrt{3})/4$ for either $i = 1$ or $2$. It is worth noting that, from this point on, the argument only depends on our choice for the auxiliary coloring $\alpha$ and not on the original coloring of $\mathbb{N}$. Thus, the following lemma holds true for the set-valued coloring $\phi$ derived from any coloring $\alpha:\mathbb{N}\to [r]$.

\begin{lem}\label{lem:r=2}
For $r=2$, the upper density of $S(\phi,1)$ or $S(\phi,2)$ is at least $f_2$. 
\end{lem}

\begin{proof}
Suppose, for the sake of contradiction, that there exists some $\epsilon>0$ and a coloring $\alpha$ such that $S(\phi,1)$ and $S(\phi,2)$ each have density at most $f_2-\epsilon$ in $[n]$ for all $n$ sufficiently large. Without loss of generality, suppose that $\alpha(1)=1$. Define $H_1$ to be the first integer with $\alpha$-color different from $1$ and, for each $i \ge 2$, define $H_i$ to be the first integer greater than $H_{i-1}$ with $\alpha$-color different from $H_{i-1}$. 

First, we claim that there exists such a coloring with the property that $H_{i+2} > 2(H_{i+1}-1)$ for all $i \geq 0$. Indeed, suppose that $i$ is the smallest non-negative integer for which $H_{i+2} \le 2 (H_{i+1}-1)$. Consider a new coloring $\alpha'$ where we change the $\alpha$-color of every integer in $[H_{i+1},H_{i+2})$ to $\alpha(H_i)$, while fixing the color of all other integers. Let $\phi'$ be the coloring associated to $\alpha'$. 
We can verify that $\phi'(x)=\phi(x)$ for all $x\le H_{i+1}-1$ and $x > 2(H_{i+2}-1)$, while 
$\phi'(x) \subseteq \phi(x) = \{\alpha(H_i),\alpha(H_{i+1})\}$ for $x\in [H_{i+1},2(H_{i+2}-1)]$. Thus, $\phi'(x) \subseteq \phi(x)$ for all $x$, so the coloring $\alpha'$ also has the property that $S(\phi',1)$ and $S(\phi',2)$ each have density at most $f_2-\epsilon$ in $[n]$ for all $n$ sufficiently large.

It therefore suffices to consider the case where there exist $1=H_0<H_1<\cdots$ such that $H_i \ge 2H_{i-1}-1$ for all $i\ge 1$ and all elements in $[H_j,H_{j+1})$ receive color $(j+1)\pmod 2$. 
Note that $1 \in \phi(x)$ if and only if $x\in \bigcup_{j \equiv 0\pmod 2} [H_j,2(H_{j+1}-1)]$ and $2\in \phi(x)$ if and only if $x\in \bigcup_{j \equiv 1\pmod 2} [H_j,2(H_{j+1}-1)]$. Let $\bar{a}_n$ be the density of $S(\phi,n\pmod 2)$ in the interval $[2(H_n-1)]$. Then 
\begin{align*}
\bar{a}_n &= \frac{\sum_{i\equiv n\,\, (\textrm{mod } 2),i\le n}(2(H_i-1) - (H_{i-1}-1))}{2(H_n-1)}\\
&=\frac{2\sum_{i\equiv n\,\, (\textrm{mod } 2),i\le n}(H_i-1) - \sum_{i\not\equiv n\,\,(\textrm{mod } 2),i\le n}(H_i-1)}{2(H_n-1)}.
\end{align*}

Let $z_n = \frac{H_{n-1}-1}{H_n-1} \le \frac{1}{2}$ and $\bar{b}_n = \bar{a}_{n-1}z_n$, noting that $\bar{b}_n$ is at most the density of $S(\phi,n-1\pmod 2)$ in the interval $[2(H_n-1)]$. 
Observe that 
\begin{align*}
\bar{a}_n &= \frac{2\sum_{i\equiv n\,\, (\textrm{mod } 2),i\le n}(H_i-1) - \sum_{i\not\equiv n\,\, (\textrm{mod } 2),i\le n}(H_i-1)}{2(H_n-1)}\\
&= \frac{2\sum_{i\equiv n\,\, (\textrm{mod } 2),i\le n-2}(H_i-1) - \sum_{i\not\equiv n\,\, (\textrm{mod } 2),i\le n-2}(H_i-1)}{2(H_{n-2}-1)} \cdot \frac{H_{n-2}-1}{H_n-1} + \frac{2(H_n-1)-(H_{n-1}-1)}{2(H_n-1)}\\
&= \bar{a}_{n-2} z_{n-1}z_n + 1 - z_n/2\\
&= \bar{b}_{n-1}z_n + 1 - z_n/2.
\end{align*} 
Thus, 
\[
(\bar{a}_n,\bar{b}_n) = \left(\bar{b}_{n-1}z_n+1-z_n/2,\bar{a}_{n-1}z_n\right).
\]
Let $B = [0,f_2 - \epsilon]^2$. 
For $S\subseteq [0,1]^2$, define 
\[
g(S) = \{(bz+1-z/2,az)\,:\,(a,b)\in S, z\in [0,1/2]\} \cap B.
\]
Since there is a coloring such that both $S(\phi,i)$ have density at most $f_2-\epsilon$ in $[n]$ for all $n$ sufficiently large, letting $(a,b) = (\bar{a}_t,\bar{b}_t)$ for $t$ sufficiently large, we have, by induction, that $(\bar{a}_{t+k},\bar{b}_{t+k}) = (\bar{b}_{t+k-1}z_{t+k}+1-z_{t+k}/2,\bar{a}_{t+k-1}z_{t+k}) \in g^{k}(S)$ for all $k\ge 1$. Thus, there is a point $(a,b)$ such that $g^k(\{(a,b)\})$ is non-empty for all $k\ge 1$. 
Let $S_0$ be the set of points $(a,b) \in B$ such that $g^k(\{(a,b)\})$ is non-empty for all $k$. For each natural number $K$, let $S_K$ be the set of points $x_0=(a_0,b_0)\in B$ for which there exists $z_k\in [0,1/2]$ for each $1\le k\le K$ such that $x_k = (a_k,b_k) = (b_{k-1}z_k+1-z_k/2,a_{k-1}z_k) \in B$. Observe that $S_0 = \bigcap_{K\ge 1}S_K$. In the following claim, we show that $S_0$ is convex and closed. 

\begin{claim} \label{clm:cnvx}
$S_0$ is convex and closed. 
\end{claim}

\begin{proof}
Since $S_0 = \bigcap_{K\ge 1}S_K$, it suffices to show that $S_K$ is convex and closed for each $K$. 

First, we show that $S_K$ is convex. Indeed, assume that $x_0=(a_0,b_0)$ and $x_0'=(a_0',b_0')$ are in $S_K$ and $y_0=(c_0,d_0)=\alpha_0 x_0+(1-\alpha_0)x_0'$ for some $\alpha_0 \in [0,1]$. As $x_0$ and $x_0'$ are in $S_K$, we have that $x_0,x_0'\in B$ and there exist $z_k$ and $z_k'$ in $[0,1/2]$ for each positive integer $k\le K$ such that $x_k = (a_k,b_k) = (b_{k-1}z_k+1-z_k/2,a_{k-1}z_k)$ and $x_k' = (a_k',b_k') = (b_{k-1}'z_k'+1-z_k'/2,a_{k-1}'z_k')$ are in $B$. Since $B$ is convex, $y_0\in B$. We will show by induction that for each positive integer $k\le K$ there exists $w_k\in [0,1/2]$ such that $y_k=(c_k,d_k)=(d_{k-1}w_k+1-w_k/2,c_{k-1}w_k)$ is a convex combination of $x_k$ and $x_k'$ and, hence, $y_k\in B$. This shows that $y_0 \in S_K$. 

We will need the following simple observation: any points $t$, $u$, $u'$, $v$, $\tilde{u}$ and $\tilde{u}'$ in $\mathbb{R}^2$ such that $v$ is on the segment between $u$ and $u'$, $\tilde{u}$ is on the segment between $t$ and $u$ and $\tilde{u}'$ is on the segment between $t$ and $u'$ have the property that the segment between $\tilde{u}$ and $\tilde{u}'$ intersects the segment between $t$ and $v$. 

The set of points $(b_{k-1}z+1-z/2,a_{k-1}z)$ for $z\in [0,1/2]$ is a segment with one endpoint at $(1,0)$ and the other endpoint at $\frac{1}{2}(b_{k-1}-1/2,a_{k-1})+(1,0)$. Similarly, the set of points $(b_{k-1}'z+1-z/2,a_{k-1}'z)$ is a segment with one endpoint at $(1,0)$ and the other endpoint at $\frac{1}{2}(b_{k-1}'-1/2,a_{k-1}')+(1,0)$. Noting that $(a,b)\mapsto \frac{1}{2}(b-1/2,a) + (1,0)$ is a linear map, we have, since $(c_{k-1}, d_{k-1})$ is a convex combination of $(a_{k-1}, b_{k-1})$ and $(a_{k-1}',b_{k-1}')$ by the induction hypothesis, that the point $\frac{1}{2}(d_{k-1}-1/2,c_{k-1})+(1,0)$ is a convex combination of $\frac{1}{2}(b_{k-1}-1/2,a_{k-1})+(1,0)$ and $\frac{1}{2}(b_{k-1}'-1/2,a_{k-1}')+(1,0)$. Therefore, by the observation above, for any $z,z' \in [0,1/2]$, the segment through $(b_{k-1}z+1-z/2,a_{k-1}z)$ and $(b_{k-1}'z'+1-z'/2,a_{k-1}'z')$ intersects the segment of points $(d_{k-1}z''+1-z''/2,c_{k-1}z'')$ with $z''\in [0,1/2]$. Thus, there exists $w_k\in [0,1/2]$ such that $y_k=(d_{k-1}w_k+1-w_k/2,c_{k-1}w_k)$ is a convex combination of $x_k$ and $x_k'$, as required.

Next, we verify that $S_K$ is closed. Let $x_0^{i}$ be a sequence of points in $S_K$ converging to $x_0$. Then we have $x_0\in B$, since $x_0^{i}\in B$ for all $i$ and $B$ is closed. Since $x_0^{i}\in S_K$, there exists $z_k^{i}\in [0,1/2]$ for $1\le k\le K$ such that $x_k^{i}=(a_k^{i},b_k^{i})= (b_{k-1}^iz_k^i+1-z_k^i/2,a_{k-1}^iz_k^i)$ is in $B$. Since $[0,1/2]^{K}$ is compact, the Bolzano--Weierstrass Theorem implies that there exists a subsequence $i_j$ such that $(z_k^{i_j})_{k\le K}$ converges to a limit $(z_k)_{k\le K}$. For $1\le k\le K$, define $x_k= (a_k,b_k) = (b_{k-1}z_k+1-z_k/2,a_{k-1}z_k)$ inductively. We now prove by induction on $0\le k\le K$ that $x_k = \lim_{j\to \infty} (a_k^{i_j},b_k^{i_j})$. Indeed, this holds for $k=0$. Furthermore, if $x_{k-1} = \lim_{j\to \infty} (a_{k-1}^{i_j},b_{k-1}^{i_j})$, then, as $\lim_{j\to \infty} z_k^{i_j} = z_k$, we have 
$$\lim_{j\to \infty} (a_k^{i_j}, b_k^{i_j}) =  \lim_{j\to \infty} (b_{k-1}^{i_j}z_k^{i_j}+1-z_k^{i_j}/2,a_{k-1}^{i_j}z_k^{i_j}) = (b_{k-1}z_k+1-z_k/2,a_{k-1}z_k) =x_k,$$
as required. Since $x_k^{i_j}\in B$ for all $j$ and $B$ is closed, we have that $x_k \in B$ for all $k\le K$. In particular, $x_0 \in S_K$. Hence, $S_K$ is closed. 
\end{proof}

For each $x=(a,b)\in S_0$, let $t(x) = (bz+1-z/2,az)$, where $z$ is the largest element of $[0,1/2]$ such that $(bz+1-z/2,az) \in S_0$. It is clear that such a $z$ exists for $x\in S_0$ by the definition of $S_0$ and the fact that $S_0$ is closed. We next show that $t(x)$ is a continuous map. For $x=(a,b)\in S_0$, there exists $z\in [0,1/2]$ such that $bz+1-z/2 \le f_2$. Thus, $b \le 2(f_2 - 3/4) < 1/2$. In particular, $S_0$ is a subset of $[0,1]\times [0,2(f_2-3/4)]$.
Define the function $\pi(x) = (a/(2a-2b+1),a/(2a-2b+1))$ for $x=(a,b)\in [0,1]\times [0,2(f_2-3/4)]$ and note that $\pi$ is continuous on its domain. 
Let $I$ be the image $\pi(S_0)$ of $S_0$, which is a closed interval consisting of points $x=(a,a)$ where $0\le a\le 1/(3-4(f_2-3/4))<1/2$. For $x=(a,b)\in \pi^{-1}(I)$, define the function $v(x) = \sup\{z\ge 0\,:\,(bz+1-z/2,az)\in S_0\}$. Observe that $(a/(2a-2b+1)-1/2,a/(2a-2b+1))=\frac{1}{2a-2b+1}(b-1/2,a)$. Thus, for all $x = (a,b) \in \pi^{-1}(I)$, 
\[
v(\pi(x)) = \sup\left\{z \ge 0\,:\,(1,0)+\frac{z}{2a-2b+1}(b-1/2,a) \in S_0\right\} = (2a-2b+1)v(x).
\]
In particular, $v(x)$ is well-defined and finite for $x\in \pi^{-1}(I)$, as, for any such $x$, there exists a point $y$ of $S_0$ for which $\pi(x) = \pi(y)$ and, since $v(y)$ is finite, $v(\pi(x))=v(\pi(y))$ is finite and so is $v(x)$. For $x\in \pi^{-1}(I)$, define $u(x) = (bv(x)+1-v(x)/2,av(x))$. We then have 
\[
u(\pi(x)) = (1,0) + \frac{v(\pi(x))}{2a-2b+1} (b-1/2,a) =  (1,0) + v(x)(b-1/2,a) = u(x).
\]
Noting that $I \subseteq \pi^{-1}(I)$, let $\tilde{v}:I \to \mathbb{R}$ be the restriction of $v$ to $I$ and $\tilde{u}: I\to \mathbb{R}^2$ the restriction of $u$ to $I$. The next claim shows that $\tilde{v}$ is continuous on $I$. 
\begin{claim}
The function $\tilde{v}$ is continuous on $I$.
\end{claim}
\begin{proof}
Recall that, for any $x=(a,a)\in I$, we have $a\le 1/(6-4f_2)<1/2$. Since $S_0 \subseteq B = [0,f_2-\epsilon]^2$, we have $1+\tilde{v}(x)(a-1/2) \ge 0$ and so $\tilde{v}(x) \le 1/(1/2-1/(6-4f_2))$ for all $x\in I$. Similarly, 
$1+\tilde{v}(x)(a-1/2) \le f_2$ and $\tilde{v}(x) \ge 2(1-f_2)$ for all $x\in I$. Thus, there exist constants $\lambda, \Lambda>0$ such that $\lambda<\tilde{v}(x)<\Lambda$ for all $x\in I$. 

Let $i_1=(a_1,a_1),\,\,i_3=(a_3,a_3)\in I$ and $i_2 = (a_2,a_2)$, where $a_2 = ca_1+(1-c)a_3$ is a convex combination of $i_1$ and $i_3$. Let $c'=\frac{c\tilde{v}(i_3)}{c\tilde{v}(i_3) +(1-c)\tilde{v}(i_1)}$. We claim that $\tilde{v}(i_2) \ge c'\tilde{v}(i_1)+(1-c')\tilde{v}(i_3)$. Let $z = c'\tilde{v}(i_1)+(1-c')\tilde{v}(i_3)$. Then 
\[
c' a_1\tilde{v}(i_1) + (1-c') a_3\tilde{v}(i_3) = \frac{\tilde{v}(i_1)\tilde{v}(i_3) (ca_1+(1-c)a_3)}{c\tilde{v}(i_3)+(1-c)\tilde{v}(i_1)} = a_2 z
\]
and
\[c' (a_1-1/2) \tilde{v}(i_1) + (1-c') (a_3-1/2) \tilde{v}(i_3) = \frac{\tilde{v}(i_1)\tilde{v}(i_3) (c(a_1-1/2)+(1-c)(a_3-1/2))}{c\tilde{v}(i_3)+(1-c)\tilde{v}(i_1)} = (a_2 -1/2) z.
\]
Therefore, writing $p_1 = (a_1\tilde{v}(i_1)-\tilde{v}(i_1)/2+1,a_1\tilde{v}(i_1))$ and $p_3 = (a_3\tilde{v}(i_3)-\tilde{v}(i_3)/2+1,a_3\tilde{v}(i_3))$, we have that $((a_2-1/2)z+1,a_2z) = c'p_1+(1-c')p_3$. Since $S_0$ is convex and $p_1,p_3\in S_0$, we thus have that $((a_2-1/2)z+1,a_2z) \in S_0$. In particular, $\tilde{v}(i_2) \ge z  = c'\tilde{v}(i_1)+(1-c')\tilde{v}(i_3)$. Hence, since $c'=\frac{c\tilde{v}(i_3)}{c\tilde{v}(i_3) +(1-c)\tilde{v}(i_1)} \ge 1 - \frac{(1-c)\Lambda}{\lambda}$,
for all points $i_1,i_2 \in I$ such that there exists $i_3$ with $i_2 = ci_1+(1-c)i_3$, we have 
\begin{equation}
\tilde{v}(i_2) \ge \left(1 - \frac{(1-c)\Lambda}{\lambda}\right)\tilde{v}(i_1). \label{eq:bound-tildev}
\end{equation}

Using this, we now show that $\tilde{v}$ is lower semi-continuous on $I$. Indeed, assume otherwise that there exists a sequence of points $x_i\in I$ converging to $x = (a,a) \in I$ with $\liminf \tilde{v}(x_i) = w < \tilde{v}(x)$. For any $\eta>0$, there exists $\delta>0$ such that if $|x_i-x|<\delta$, then we can write $x_i = cx+(1-c)y$ with $y \in I$ and $c>1-\eta$. By (\ref{eq:bound-tildev}), we therefore have 
\[
\tilde{v}(x_i) \ge \left(1 - \frac{(1-c)\Lambda}{\lambda}\right)\tilde{v}(x) \ge \left(1 - \frac{\eta\Lambda}{\lambda}\right)\tilde{v}(x).
\]
But, for $\eta$ sufficiently small, this contradicts our assumption that $\lim \inf \tilde{v}(x_i) = w < \tilde{v}(x)$. 

Next, we show that $\tilde{v}$ is upper semi-continuous on $I$. Indeed, assume otherwise that there is a sequence of points $x_i\in I$ converging to $x = (a,a) \in I$ with $\limsup \tilde{v}(x_i) = w > \tilde{v}(x)$. We can then extract a subsequence $x_{i_j}$ for which $\tilde{v}(x_{i_j})$ converges to $w$. 
But then, by the fact that $S_0$ is closed, $(aw+1-w/2,aw) \in S_0$ and, hence, $\tilde{v}(x) \ge w$, a contradiction. 

Therefore, since $\tilde{v}$ is both lower and upper semi-continuous on $I$, it is continuous on $I$.
\end{proof}

Since $\tilde{v}$ is continuous on $I$, we obtain that $\tilde{u}$ is also continuous on $I$. Then, by the continuity of $\pi$ on $S_0$ and the fact that $u(x)=u(\pi(x))=\tilde{u}(\pi(x))$, $u$ is continuous on $S_0$ and, hence, $v$ is continuous on $S_0$. Thus, $x \mapsto t(x) = (1,0) + \min(1/2,v(x)) (b-1/2,a)$ for $x=(a,b)$ is also continuous on $S_0$. 

Since $t$ is a continuous map from $S_0$ to itself and $S_0$ is bounded, closed and convex, we may apply the Brouwer fixed-point theorem to conclude that $t$ has a fixed point $x_0$. Let $x_0 = (a_0,b_0)$. We then have, for some $z\in [0,1/2]$, that
\[
b_0z+1-z/2=a_0,\qquad a_0z=b_0.
\]
Thus, 
\[
a_0 = \frac{1-z/2}{1-z^2} \ge \inf_{z\in [0,1/2]} \frac{1-z/2}{1-z^2} \ge f_2.
\]
However, this is a contradiction, since $a_0 \le f_2-\epsilon$ for $x_0=(a_0,b_0)\in S_0$. 
\end{proof}

\section{Concluding remarks}

We conjecture that our upper bound for $c_r$ is also tight for three or more colors.

\begin{conj}
For any integer $r \geq 3$, $c_r$ is equal to
$$\left(1-\frac{1}{2b_0}\right)\left(1+\frac{1}{2rb_0-r}\right),$$ 
where $b_0$ is the unique root of the polynomial $b^r-2rb+r-1$ with $b > 1$.
\end{conj}

We have explored this conjecture in some detail ourselves, but were unable to establish the optimality of our upper bound for $c_r$ without additional assumptions. For instance, it seems that our methods do apply if the auxiliary coloring $\alpha$ defined in Section~\ref{sec:lower} is assumed to be \emph{cyclic}, by which we mean that the $i$th monochromatic interval in $\alpha$ has color $i \pmod{r}$ for all $i \geq 1$. 
Since every two-coloring of the positive integers is automatically cyclic in this sense, this restriction does not hamper us in that case.

\end{document}